\documentclass[a4paper,11pt,reqno]{amsart}

\usepackage[utf8]{inputenc}
\usepackage[T1]{fontenc}
\usepackage{amsmath,graphics}
\usepackage{amssymb}
\usepackage{amsfonts}
\usepackage{mathtools}
\usepackage{latexsym}
\usepackage[dvips]{graphicx}
\usepackage{dirtytalk}
\usepackage{chngcntr}
\usepackage{amsthm}
\usepackage{xfrac}
\numberwithin{equation}{section}

\usepackage[usenames,dvipsnames]{color}
\usepackage[colorlinks=true,linkcolor=Red,citecolor=Green]{hyperref}

\usepackage[
backend=bibtex,
giveninits=true,
style=alphabetic,
maxbibnames=99,mincitenames=1, maxcitenames=4, isbn = false, doi = false, url = false]{biblatex}
\DeclareFieldFormat[article,periodical]{volume}{\mkbibbold{#1}} 

\renewbibmacro{in:}{
	\ifboolexpr{
		test {\ifentrytype{article}}
		or
		test {\ifentrytype{inproceedings}}
	}{}{\printtext{\bibstring{in}\intitlepunct}}
}

\iffalse \usepackage[
backend=bibtex,
giveninits=true,
maxbibnames=99,mincitenames=1, maxcitenames=4, isbn = false, doi = false, url = false, isbn =]{biblatex}
\DeclareFieldFormat[article,periodical]{volume}{\mkbibbold{#1}} 

\renewbibmacro{in:}{%
		\ifboolexpr{%
				test {\ifentrytype{article}}%
				or
				test {\ifentrytype{inproceedings}}%
			}{}{\printtext{\bibstring{in}\intitlepunct}}%
	}\fi 

\usepackage{xcolor}

\usepackage{enumerate}
\usepackage{enumitem}

\theoremstyle{plain}
\newtheorem{thm}{Theorem}[section]
\newtheorem{propo}[thm]{Proposition}
\newtheorem{lem}[thm]{Lemma}

\theoremstyle{definition}

\newtheorem{rem}[thm]{Remark}

\title[Spectral gap of random surface covers]{Spectral gap of random covers of negatively curved noncompact surfaces}

\author[J.~Moy]{Julien Moy}

\address{Julien Moy \\
	Laboratoire de math\'ematiques d'Orsay \\
	Universit\'e Paris-Saclay, 91405 Orsay Cedex\\France.}
\email{julien.moy@universite-paris-saclay.fr}



\begin{document}
	
	\maketitle
		\begin{abstract}
		Let $(X,g)$ be a complete noncompact geometrically finite surface with pinched negative curvature $-b^2\leq K_g \leq -1$. Let $\lambda_0(\widetilde{X})$ denote the bottom of the $L^2-$spectrum of the Laplacian on the universal cover $\widetilde{X}$. We show that a uniformly random degree-$n$ cover $X_n$ of $X$ has no eigenvalues below $\lambda_0(\widetilde{X})-\varepsilon$ other than those of $X$ and with the same multiplicity, with probability tending to $1$ as $n\to \infty$. This extends a result of Hide--Magee to metrics of pinched negative curvature. 
	\end{abstract}
	\tableofcontents
	\section{Introduction}
Let $(X,g)$ be a complete, connected, noncompact Riemannian surface, with pinched sectional curvature $-b^2\le K_g\le -1$ for some $b\ge 1$. We assume that $X$ is \emph{geometrically finite} in the sense of Bowditch \cite{bowditch1995geometrical}. Denote by $(\widetilde X,\tilde g)$ the universal cover of $X$, endowed with the lifted metric. Then $\widetilde X$ is a Cartan--Hadamard manifold and the fundamental group $\Gamma$ of $X$ acts by deck transformations on $\widetilde X$, so that $X$ identifies with the quotient $\Gamma\backslash \widetilde X$. Since $X$ is noncompact and geometrically finite, it follows from \cite[Proposition 5.5.1]{bowditch1995geometrical} that $\Gamma$ is a free group of finite rank.

Motivated by the work of Hide--Magee \cite{HM} in constant curvature, and our previous work with Hide and Naud \cite{hide2025spectralgap} where we dealt with the compact case, we investigate the behavior of the Laplacian spectrum below the threshold $\lambda_0(\widetilde X)$ on random Riemannian coverings of $X$ of large degree.

Let $[n]=\{1,\ldots,n\}$, and denote by $\mathbf S_n$ the symmetric group over $[n]$. Given a homomorphism $\phi:\Gamma\to \mathbf S_n$, one defines a degree-$n$ covering of $X$ by setting
\[X_\phi:=\Gamma\backslash(\widetilde X\times [n]),\]
where $\Gamma$ acts on the product by $\gamma.(x,j)=(\gamma x,\phi[\gamma](j))$. Since $\Gamma$ is finitely generated, the set of homomorphisms $\mathrm{Hom}(\Gamma,\mathbf S_n)$ is finite, hence we can endow it with the uniform probability measure. Note that if $\Gamma=\langle \gamma_1,\ldots,\gamma_r\rangle$ then $\phi[\gamma]$ is simply determined by the choices of $\phi[\gamma_i]$. This gives a model of random degree-$n$ covers of $X$.
Let $\lambda_0(\widetilde X)$ denote the bottom of the $L^2$-spectrum of the Laplacian on $\widetilde X$. For any finite degree cover $X_\phi\to X$, the spectrum of $X_\phi$ below $\lambda_0(\widetilde X)$ consists only of discrete eigenvalues (see \S \ref{subsec: splitting}). Our main result is as follows.

\begin{thm}\label{thm: main} For any $\varepsilon>0$, with probability tending to $1$ as $n\to +\infty$, one has
	\[\sigma\big(\Delta_{X_\phi}\big)\cap \big[0,\lambda_0(\widetilde X)-\varepsilon\big]=\sigma\big(\Delta_{X}\big)\cap \big[0,\lambda_0(\widetilde X)-\varepsilon\big],\]
	where the multiplicities coincide on both sides.\end{thm}
In the case where $X$ is non-compact and with constant curvature, Theorem \ref{thm: main} was proved by Hide--Magee \cite{HM}\footnote{The geometrically finite case is not stated in \cite{HM} but the proof can be easily extended to also deal with funnel ends.}. When $X$ is a closed surface, by \cite{HM,louder2022strongly} in constant curvature and by \cite[Theorem 1.3]{hide2025spectralgap} in variable negative curvature, a recent result of Magee, Puder and van Handel proves Theorem \ref{thm: main}. When $X$ is a Schottky surface, there is a stronger result due to Calder\'{o}n, Magee and Naud \cite{MCN} which goes beyond $L^2$-spectral gaps.
\subsection{Related work}
\subsubsection*{Random covers}

The spectral gap of random covers of hyperbolic surfaces were first studied by Magee--Naud \cite{MN2019} for Schottky surfaces and Magee--Naud--Puder \cite{MNP} for closed surfaces where a relative spectral gap of $\frac{3}{16}-\varepsilon$ was obtained. In the hyperbolic case, one has $\widetilde{X}=\mathbb{H}$ and $\lambda_0\left(\mathbb{H}\right)=\frac{1}{4}$ and it is a theorem of Huber \cite{Huber} that $\frac{1}{4}$ is the asymptotically optimal spectral gap in this setting. \\
The bound $\frac{1}{4}-\varepsilon$ was obtained by Hide--Magee for non-compact finite-area hyperbolic surfaces \cite{HM}. In the Schottky setting, the conjecturally optimal spectral gap was obtained by Calder\'{o}n, Magee and Naud in \cite{MCN}. This progress relied on strong convergence of permutation representations (see \cite{Magee_survey} for more background and an overview of recent developments and \S\ref{Sec-strong}  for definitions). 
A breakthrough result of Bordenave--Collins \cite{bordenave2019eigenvalues} (c.f. Theorem \ref{prop: strong conv}),  which is a key input in the current article, showed that uniformly random $\phi \in \text{Hom}\left(\mathbf{F}_r,\mathbf S_n\right)$ composed with the standard irreducible representation of $\mathbf S_n$ strongly converge to the regular representation (see \S\ref{Sec-strong}). As an important consequence, this resolved a conjecture of Friedman \cite{FriedmanRelative} that for any fixed finite graph $\mathcal{G}$, with high probability, a uniformly random cover of large degree has no eigenvalues of absolute value greater than $\rho(\tilde{\mathcal{G}})+\varepsilon$ where $\rho(\tilde{\mathcal{G}})$ denotes the spectral radius of the adjacency operator on the universal cover $\tilde{\mathcal{G}}$. We mention that recently, a new remarkable proof was given by Chen, Garza-Vargas, Tropp and van Handel \cite{Ch.Ga.Tr.va2024}. \\ It was observed \cite{HM,MCN} that the strong convergence phenomenon can be utilised to prove near optimal spectral gaps for covers of hyperbolic manifolds, and for closed surfaces of variable negative curvature \cite{hide2025spectralgap}. Since the fundamental group of any non-compact surface is free, the articles \cite{HM,MCN} relied on the results of Bordenave--Collins \cite{bordenave2019eigenvalues} as input.
The absence of an analogous result for uniformly random permutation representations of surface groups presented a significant barrier to obtaining Theorem \ref{thm: main} for the closed case. This was very recently overcome in a breakthrough of Magee, Puder and van Handel \cite{Ma.Pu.vH2025}.
\subsubsection*{Weil--Petersson model}
Recently, it was shown in an impressive work by Anantharaman and Monk \cite{AnMo3} that a random large genus hyperbolic surface sampled with the Weil--Petersson probability measure from the moduli space of genus-$g$ hyperbolic surfaces has a spectral gap of at least $\frac{1}{4}-\varepsilon$ with high probability. This improved on previous results of $\frac{1}{4}\left(\frac{\log(2)}{2\pi+\log(2)}\right)^2\approx 0.02$ due to Mirzakhani \cite{Mi2013}, $\frac{3}{16}-\varepsilon$ obtained independently by Wu--Xue \cite{WX} and Lipnowski--Wright \cite{LW}, and $\frac{2}{9}-\varepsilon$ obtained by Anantharaman and Monk \cite{AnMo2}.

\subsubsection*{Random unitary bundles} Let $X$ be a noncompact surface with fundamental group $\Gamma$ freely generated by $\gamma_1,\ldots,\gamma_r$. For each $n\ge 1$, one can define a random unitary representation $\rho:\Gamma\to \mathbf{U}(n)$ by sampling each $\rho(\gamma_i)$ at random with the normalized Haar measure on $\mathbf{U}(n)$. To any such unitary representation is associated a flat bundle $E_\rho$ of rank $n$ over $X$. One can then consider the flat Laplacian $\Delta_{E_\rho}$ acting on sections of $E_\rho$. Zargar \cite{zargar2022random} showed that when $X$ is a noncompact finite-area hyperbolic surface, for any $\varepsilon>0$, one has
\[\lambda_0(\Delta_{E_\rho})\ge \frac 14-\varepsilon,\]
with probability tending to $1$ as $n\to +\infty$, where $\lambda_0(\Delta_{E_\rho})$ denotes the bottom of the spectrum of $\Delta_{E_\rho}$. We believe that the heat kernel techniques used in the present paper are robust enough to extend this result to geometrically finite surfaces of pinched negative curvature (replacing $\frac 14$ by $\lambda_0(\widetilde X)$).

	\subsection{Overview of the proof} As in \cite{hide2025spectralgap} which dealt with the compact case, we obtain a lower bound on the first new eigenvalue of $\Delta_{X_\phi}$ by obtaining an upper bound on the operator norm of the heat operator $\exp(-t\Delta_{X_\phi})$, acting on the space $L^2_{\rm new}(X_\phi)$ of functions that are orthogonal to all lifts of functions from $X$. To do so, similarly to \cite{HM}, we have to separate the contributions from the ends of the surface from those of the interior. More precisely, we fix a base point $o\in X$ then let 
	\[\mathcal C_R=\{x\in X~:~d(x,o)> R\},\]
	which is nothing but the complement of the closed ball of radius $R$ centered at $o$. Let $\chi_{\mathcal C_R}$ denote the indicator function of $\mathcal C_R$, and let $\chi_{\mathcal C_R^\phi}$ denote its lift to the cover $X_\phi\to X$. We then decompose the heat operator as
	\[\mathrm{e}^{-t\Delta_{X_\phi}}=\mathrm{e}^{-t\Delta_{X_\phi}}(1-\chi_{\mathcal C_R^\phi})+\mathrm{e}^{-t\Delta_{X_\phi}}\chi_{\mathcal C_R^\phi}.\]
	We refer to the first term as \emph{interior part} and the second as the \emph{end part}. These two contributions are analyzed separately. \medskip
	
\emph{The end part.}	We show (Proposition \ref{prop: heat ends}) that if $r\le R$ are two radii with $R\ge \max(2r,t^2)$, then the end part $\exp({-t\Delta_{X_\phi}})\chi_{\mathcal C_R^\phi}$ is well approximated by the operator
	\begin{equation} \label{heat operator in end} \chi_{\mathcal C_r^\phi} \exp(-t\Delta_{\mathcal C_r^\phi})\chi_{\mathcal C_R^\phi}. \end{equation}
	Here $\exp(-t\Delta_{\mathcal C_r^\phi})$ is the heat operator associated with the Dirichlet Laplacian in the end $\mathcal C_r^\phi$. The proof is essentially based on Gaussian upper bounds on the heat kernel and control on the decay of the injectivity radius in the ends of the manifold. This result provides a quantitative formulation of the probabilistic intuition that a Brownian motion starting at distance $\ge R$ from $o$ should remain at distance $\ge r$ from $o$ for all times $t\ll R-r$. Now, when $r$ is large, the bottom of the spectrum of $\Delta_{\mathcal C_r^\phi}$ approaches $\lambda_{\rm ess}(X)$. This allows controlling the operator norm of the end part in terms of the bottom of the essential spectrum $\lambda_{\rm ess}(X)$ of the base surface. \medskip
	
	\emph{The interior part.} To deal with the interior part, analogously to \cite{HM} who worked with the resolvent, we use the fact that the operator $\mathrm{e}^{-t\Delta_{X_\phi}}(1-\chi_{\mathcal C_R^\phi})|_{L^2_{\rm new}(X_\phi)}$ is unitarily conjugated to an operator of the form
	\[\sum_{\gamma\in \Gamma} a_\gamma\otimes \rho_\phi(\gamma)\]
	acting on $L^2(\mathcal F)\otimes V_n^0$, where $\mathcal F$ is a Dirichlet fundamental domain for the action of $\Gamma$ and $\rho_\phi$ is the composition
	\[\Gamma\overset{\phi}\longrightarrow \mathbf S_n\overset{\mathrm{std}_{n-1}}{\longrightarrow} \mathcal U(V_n^0),\]
	where $(\operatorname{std}_{n-1},V_n^0)$ is the standard $(n-1)$-dimensional representation of $\mathbf S_n$. By the aforementioned theorem of Bordenave--Collins  (c.f. Theorem \ref{prop: strong conv}), it follows that for any $\varepsilon>0$, with high probability as $n\to +\infty$, the random operator satisfies 
	\begin{equation} \label{eq: intro sketch strong conv} \Big\|\sum_{\gamma\in \Gamma} a_\gamma\otimes \rho_\phi(\gamma)\Big\|_{L^2(\mathcal F)\otimes V_n^0}\le \Big\|\sum_{\gamma\in \Gamma} a_\gamma\otimes \rho_\infty(\gamma)\Big\|_{L^2(\mathcal F)\otimes \ell^2(\Gamma)}+\varepsilon,\end{equation}
	where $(\rho_\infty,\ell^2(\Gamma))$ denotes the regular representation of $\Gamma$. It turns out that the operator on the right-hand side of \eqref{eq: intro sketch strong conv} is explicitly related to the heat operator on the universal cover $\mathrm{e}^{-t\Delta_{\widetilde X}}$, whose operator norm is just $\mathrm{e}^{-t\lambda_0(\widetilde X)}$. It allows us to show that with high probability as $n\to +\infty$, the interior part of the heat propagator $\mathrm{e}^{-t\Delta_{X_\phi}}$ acting on $L^2_{\rm new}(X_\phi)$ has operator norm controlled by $\mathrm{e}^{-t\lambda_0(\widetilde X)}$. \medskip 
	
	A result of Ballmann--Polymerakis \cite{BallmannEssentialBottom} ensures that $\lambda_{\rm ess}(X)\ge \lambda_0(\widetilde X)$ whenever $X$ is geometrically finite. After combining the contributions from the interior part and the end part, this allows us to show that for any $\varepsilon>0$ and $t$ large enough depending on $\varepsilon$, with probability tending to $1$ as the degree $n$ goes to $+\infty$, we have
	\[\big\|\exp(-t\Delta_{X_\phi})\big\|_{L^2_{\rm new}(X_\phi)}<\exp(-t(\lambda_0(\widetilde X)-\varepsilon)).\]
	This implies Theorem \ref{thm: main}.
	
	\subsection*{Notations} We use the letters $C,c>0$ to denote some positive constants that may vary from line to line. Unless stated otherwise, these constants will only depend on the constant $b$ appearing in the pinching condition and the fixed base point $o\in X$. In some places, we may write $C=C(\bullet)$ to stress the dependencies.
	
	If $A:H\to H$ is a continuous operator on a Hilbert space, we use $\|A\|_H$ to denote the operator norm of $A$. If $A$ is an integral operator, we will often denote its kernel by $A(x,y)$.
	
	\subsection*{Acknowledgments} I want to thank Will Hide for our conversations about this work and his feedback on a previous version of this article.

	\section{Geometric preliminaries}
	
	\subsection{Comparison estimates} Let $(M,g)$ be a smooth Riemannian surface. We start by recalling some comparison estimates for the distance function and the injectivity radii. These are used in the next subsection to build some smooth cutoff functions that localize in balls $B(x,R)$, with controlled estimates with respect to $x$ and $R$.
	
	\begin{lem}[{\cite[Lemma 2.3.2]{jost1984harmonic}}]\label{lem: laplacianDistance} Let $B(x,\rho)$ be a ball in $M$ which is disjoint from the cut locus of $x$. We assume that
		\[-b^2\le K_g\le 0.\]
		Let $f(y)=\frac 12d(x,y)^2$. Then,
		\[|\nabla f(y)|=d(x,y), \qquad 0\le \Delta f(y)\le 2bd(x,y)\coth(bd(x,y)).\]
		In particular
		\[|\Delta f|\le C(b)(d(x,y)+1).\] \end{lem}
	
	\begin{lem}[Decay of injectivity radius]\label{lem: injrad}Assume $(M,g)$ is complete and $-b^2\le K_g\le 0$. Fix $o\in M$. Then, there is a constant $C=C(o)$ such that for all $x\in X$,
		\[\operatorname{inj}(x)\ge C \exp(-bd(x,o)).\]\end{lem}
	This result can be found in \cite[Corollary 2.4]{MULLER2007158}, and follows from the injectivity radius estimates of Cheeger--Gromov--Taylor \cite{cheeger1982}. 
	\begin{rem}We will use repeatedly the fact that if $p:M_1\to M_0$ is a Riemannian covering, then $\operatorname{inj}(x)\ge \operatorname{inj}(p(x))$ for all $x\in M_1$.\end{rem}\medskip
	
	We also have the following comparison estimates for the volume of balls, known as the G\"unther and Bishop--Gromov comparison theorems.
	\begin{lem}[Volume comparison]\label{lem: volume comparison} Assume $(M,g)$ is complete and $-b^2\le K_g\le -1$. For $x\in M$, let $V(x,r)$ denote the volume of the ball $B(x,r)$. Then,
	\[\forall r\le \operatorname{inj}(x), \ V(x,r)\ge \pi r^2,\]
	and
	\[\forall r\ge 0, \ V(x,r)\le C(b)\exp(br).\]\end{lem}
	
	\subsection{Smooth cutoff functions} We now restrict to the case $M=X=\Gamma\backslash \widetilde X$, where $\widetilde X$ is a Cartan--Hadamard manifold with pinched negative curvature. Since the distance function on $X$ is not smooth when $\operatorname{inj}(X)<+\infty$, we have to resort to a slightly technical construction to produce smooth cutoff functions that localize in large balls, with controlled estimates on the derivatives.
	
	\begin{lem}\label{lem: cutoff} Assume $-b^2\le K_g\le -1$ everywhere. Let $x_0\in X$ and define $\rho(x_0):=\min(1,\operatorname{inj}(x_0))$. Then, for every $R\ge 1$, there exists a smooth cutoff $\omega\in C^\infty(X,[0,1])$ such that $\omega\equiv 1$ in $B(x_0,R)$ and $\omega\equiv 0$ on $X\backslash B(x_0,2R)$. Moreover, we have the estimates on the gradient and Laplacian of $\omega$:
		\[|\nabla \omega|\le C\rho(x_0)^{-2}\exp(2bR), \qquad |\Delta \omega|\le C\rho(x_0)^{-4}\exp(4bR).\]
		Here $C$ is independent of $x_0$ and $R$.\end{lem}
		 In the course of the proof we will use multiple times the composition formula 
		\[\Delta(\psi\circ f)=(\psi'\circ f)\Delta f+(\psi''\circ f)|\nabla f|^2\]
		for $\psi\in C^\infty(\mathbf R)$ and $f\in C^\infty(X,\mathbf R)$.
	\begin{proof} Let $\tilde x_0\in \widetilde X$ denote an arbitrary lift of $x_0$. Let $\psi\in C^\infty(\mathbf{R},[0,1])$ be such that $\psi\equiv 1$ on $[0,1]$ and $\psi\equiv 0$ on $[2,+\infty)$. Define
		\begin{equation}\label{eq: somme gamma dist} \widetilde \psi_R(\tilde x):=\sum_{\gamma\in \Gamma} \psi\Big(\frac{d(\gamma \tilde x_0,\tilde x)^2}{R^2}\Big). \end{equation}
		
		For all $\tilde x\in \widetilde X$, the sum \eqref{eq: somme gamma dist} contains at most $C\rho(x_0)^{-2}\exp(2bR)$ nonzero terms for some constant $C=C(b)$. This follows from the fact that the balls $\{B(\gamma \tilde x_0,\rho(x_0))\}_{\gamma\in \Gamma}$ are disjoint along with the volume estimates
		\[V(\gamma \tilde x_0,\rho(x_0))\ge \rho(x_0)^2, \qquad V(\tilde x,2R)\le C(b)\exp(2bR).\]
		Using Lemma \ref{lem: laplacianDistance}, we estimate
		\[\begin{array}{ll} |\nabla \widetilde \psi_R| & \le \displaystyle \sum_{\gamma\in \Gamma} \frac{1}{R^2} \|\psi'\|_{L^\infty} |\nabla d(\tilde x,\gamma\tilde x_0)^2|\mathbf 1_{d(\tilde x,\gamma\tilde x_0)\le 2R} \\ & \le  \displaystyle C \displaystyle  \sum_{\gamma\in \Gamma} \mathbf 1_{d(\tilde x,\gamma\tilde x_0)\le 2R}\\ & \le C \rho(x_0)^{-2}\exp(2bR),\end{array}\]
		and
			\[\begin{array}{ll} |\Delta \widetilde \psi_R| & \le \displaystyle \sum_{\gamma\in \Gamma} \Big(\frac{1}{R^2} \|\psi'\|_{L^\infty} |\Delta d(\tilde x,\gamma\tilde x_0)^2|+\frac{1}{R^4}\|\psi''\|_{L^\infty} |\nabla d(\tilde x,\gamma \tilde x_0)^2|^2\Big)\mathbf 1_{d(\gamma \tilde x_0,\tilde x)\le 2R} \\ & \le  \displaystyle C \sum_{\gamma\in \Gamma} \mathbf 1_{d(\gamma \tilde x_0,\tilde x)\le 2R} \\ & \le C \rho(x_0)^{-2}\exp(2bR).\end{array}\]
		Since $\widetilde \psi_R$ is a smooth $\Gamma$-invariant function, it descends to a smooth function $\psi_R\in C^\infty(X)$ that satisfies the same estimates. Moreover,
		\[\mathbf 1_{d(x_0,x)\le R}\le \psi_R(x)\le C\rho(x_0)^{-2}\exp(2bR)\mathbf 1_{d(x_0,x)\le 2R}.\]
		Now, let $\eta\in C^\infty(\mathbf{R},[0,1])$ be a smooth function such that $\eta\equiv 1$ on $[1,+\infty)$ and $\eta\equiv 0$ in $[0,1/2]$, and define
		\[\omega(x):=\eta(\psi_R(x)).\]
		Then $|\nabla \omega|\le C\rho(x_0)^{-2}\exp(2bR)$ and $|\Delta \omega|\le C\rho(x_0)^{-4}\exp(4bR)$. Moreover, $\omega\equiv 1$ on $B(x_0,R)$ and $\omega\equiv 0$ on $ X\backslash B(x_0,2R)$. Thus, $\omega$ satisfies the required properties.
	\end{proof}
	
	\subsection{Dirichlet fundamental domains}\label{subsec: dirichlet domain} As in \cite{HM}, it is convenient for our application to introduce a specific fundamental domain for the action of $\Gamma$. Fix $o\in X$, let $\tilde o\in \widetilde X$ be a lift of $o$ and define
	\[\mathcal F:=\big\{x\in \widetilde X\ | \ d(x,\tilde o)=\inf_{\gamma\in \Gamma} d(\gamma x,\tilde o)\big\}.\]
	We list below some basic properties of $\mathcal F$.
	\begin{propo}\label{prop: fundamental domain}Let $\mu$ denote the Riemannian measure. The following holds:
		\begin{enumerate}
			\item \label{item1} $\mathcal F$ is closed.
			\item \label{item2} $\bigcup_{\gamma\in \Gamma} \gamma\mathcal F=\widetilde X$.
			\item \label{item3} $\mu(\partial \mathcal F)=0$.
			\item \label{item4} $\operatorname{int}(\mathcal F)\cap \operatorname{int}(\gamma\mathcal F)=\varnothing $ and $\mu(\gamma \mathcal F\cap \mathcal F)=0$ for all $\gamma\neq \mathrm{Id}$.
			\item \label{item5} $\mathcal F$ is locally finite: for any compact set $K\subset \widetilde X$, there are finitely many $\gamma\in \Gamma$ such that $\gamma K\cap \mathcal F\neq \varnothing$.
			\item \label{item6} For any $f\in L^1(X)$ one has
			\[\int_X f(x)\mathrm{d}x=\int_{\mathcal F} (f\circ \pi(\tilde x))\mathrm{d}\tilde x.\]
			\item \label{item7} For any $f\in L^1(\widetilde X)$ one has
			\[\int_{\widetilde X} f(\tilde x)\mathrm{d}\tilde x=\sum_{\gamma\in \Gamma} \int_{\mathcal F} f(\gamma \tilde x) \mathrm{d}\tilde x.\]
	\end{enumerate}\end{propo}
	
	\begin{proof} All these results are standard in the constant curvature case.

		\ref{item1}. Letting $H(x_1,x_2):=\{x\in \widetilde X\ | \ d(x,x_1)\le d(x,x_2)\}$, one can write
		\begin{equation} \label{eq: F as intersec}  \mathcal F=\bigcap_{\gamma\in \Gamma} H(\tilde o,\gamma \tilde o).\end{equation}
		Now each of the subsets $H(\tilde o,\gamma \tilde o)$ is closed thus $\mathcal F$ is closed. 
		
		\ref{item2}. Let $x\in \widetilde X$. Since $\Gamma$ is discrete, there is some $\gamma\in \Gamma$ such that $d(\gamma x,\tilde o)=\operatorname{inf}_{\gamma\in \Gamma} d(\gamma x,\tilde o)$, implying $\gamma x\in \mathcal F$. Hence, $\mathcal F$ contains at least one copy of each element of $\widetilde X$.
		
		\ref{item3}. When $x_1\neq x_2$, the boundary $\partial H(x_1,x_2)$ is given by the \say{bisector}
		\[\partial H(x_1,x_2)=\{x\in \widetilde X\ | d(x_1,x)=d(x_2,x)\}.\]
		In other words, $\partial H(x_1,x_2)$ is the zero set of the smooth function $f:x\mapsto d(x_1,x)^2-d(x_2,x)^2$. Note that the gradient of $f$ is never vanishing because
		\[\nabla f(x)=2(\exp_x^{-1}(x_1)-\exp_x^{-1}(x_2)),\]
		and the exponential map $\exp_x:T_x\widetilde X\to \widetilde X$ is a global diffeomorphism because $\widetilde X$ is a Cartan--Hadamard manifold. It follows that $\partial H(x_1,x_2)$ is a smooth $1$-dimensional submanifold of $\widetilde X$. Now, \eqref{eq: F as intersec} implies
		\[\partial \mathcal F\subset \bigcup_{\gamma\in \Gamma}\partial H(\tilde o,\gamma \tilde o).\]
		Thus, $\partial \mathcal F$ is contained in the countable union of subsets of measure $0$, implying $\mu(\partial \mathcal F)=0$. We point out that in variable negative curvature, bisectors need not be geodesic lines.
		
		\ref{item4}. See e.g. \cite[Theorem 37.4.18]{Voight20}.
		
		\ref{item5}. See \cite[Corollary 3.3.2]{bowditch1995geometrical}.
		
		\ref{item6}. Arguing as in \cite[Lemma 37.4.13]{Voight20}, one can show that locally, $\mathcal F$ only intersects a finite number of the sets $H(\tilde o,\gamma \tilde o)$. It implies that the interior of $\mathcal F$ is given by
		\[\operatorname{int}(\mathcal F)=\big\{x\in \widetilde X\ | \ d(x,\tilde o)<d(\gamma x,\tilde o) \ \text{for all $\gamma\neq \mathrm{Id}$}\big\}.\]
		Then, the projection map $\pi:\widetilde X\to X$ induces a measure-preserving diffeomorphism $\pi:\operatorname{int}(\mathcal F)\to \pi(\operatorname{int}(\mathcal F))$. Moreover, $\mu(X\backslash \pi(\operatorname{int}(\mathcal F)))=0$ and $\mu(\partial F)=0$. We infer that for any $f\in L^1(X)$,
		\[\int_X f(x)\mathrm{d}x=\int_{\pi(\operatorname{int}(\mathcal F))} f(x)\mathrm{d}x=\int_{\operatorname{int}(\mathcal F)}f(\pi(\tilde x))\mathrm{d}\tilde x=\int_{\mathcal F} f(\pi(\tilde x))\mathrm{d}\tilde x.\]
		
		\ref{item7}. It follows similarly from the points (\ref{item2}) and (\ref{item4}) of the Lemma.
	\end{proof}
	
	The Dirichlet fundamental domain $\mathcal F$ enjoys the following property from \cite{HM} which relies on its local finiteness, and will allow us to apply strong convergence results in \S \ref{sec: proof main}. Although the result is proved in \cite{HM} for surfaces of constant curvature, it extends without difficulty to the geometrically finite case.
	\begin{lem}[{\cite[Lemma 5.4]{HM}}] \label{lem: finite number of terms} For any compact set $K\subset \mathcal F$ and $R>0$, there is another compact set $K'\subset \mathcal F$ and a finite set $S\subset \Gamma$ depending on $K$ and $R$ such that for any $x,y\in \widetilde X$ satisfying
		\[x\in \mathcal F, ~y\in \bigcup_{\gamma\in \Gamma}\gamma(K), \qquad d(x,y)\le R,\]
		one has
		\[x\in K', ~ y\in \bigcup_{\gamma\in S} \gamma(K).\]\end{lem}

	\section{Laplacian and heat kernel on Riemannian manifolds} We recall from the book of Grigor'yan \cite{grigoryan2009heat} some aspects of the spectral theory of the Laplacian and the heat kernel on Riemannian manifolds. We provide heat kernel estimates outside large balls, then study the heat operator in the ends of a geometrically finite manifold and its covers.
	
	\subsection{The Dirichlet Laplacian} Let $(M,g)$ be a smooth Riemannian manifold. We endow $M$ with the Riemannian measure $\mathrm{d}x$. The Sobolev space $H^1(M)$ consists of functions $u\in L^2(M)$ whose weak-gradient $\nabla u$ also lies in $L^2(M)$. It is equipped with the norm
	\[\|u\|_{H^1(M)}:=\int_{M} |u(x)|^2+|\nabla u(x)|^2\mathrm{d}x.\]
	We define $H^1_0(M)$ as the closure of $C_{\rm comp}^\infty(M)$ for the norm $\|\cdot\|_{H^1(M)}$. Then, we set
	\[H^2_0(M):=\{u\in H^1_0(M)~:~ \Delta u\in L^2(M)\}.\]
	The operator $\Delta_{|H^2_0(M)}$ is then a densely defined, self-adjoint operator on $L^2(M)$ (see \cite[Theorem 4.6]{grigoryan2009heat}), known as the \emph{Dirichlet Laplacian}. Note that when $(M,g)$ is complete, the operator $\Delta_{|C^\infty_{\rm comp}(M)}$ admits a unique self-adjoint extension. From now on, we will just denote $\Delta$ for the Dirichlet Laplacian, or $\Delta_M$ when we wish to stress the dependence on the manifold $(M,g)$.
	The spectrum and essential spectrum of $\Delta_M$ are denoted by $\sigma(\Delta_M)$ and $\sigma_{\rm ess}(\Delta_M)$ respectively. We set
	\[\lambda_0(M):=\operatorname{inf}\sigma(\Delta_M), \qquad \lambda_{\rm ess}(M):=\operatorname{inf}\sigma_{\rm ess}(\Delta_M).\]
	One has the following variational characterization of $\lambda_0(M)$ (see e.g. \cite[Theorem 10.8]{grigoryan2009heat}). For $f\in C^\infty_{\rm comp}(M)$, define the \emph{Rayleigh quotient}
	\[\operatorname{Ray}(f,M):= \frac{\int_M |\nabla f|^2}{\int_M |f|^2}.\]
	Then,
	\[\lambda_0(M)=\underset{f\in C^\infty_{\rm comp}(M), \ f\ge 0}{\inf} \operatorname{Ray}(f,M),\]
	in particular $\lambda_0(M)\ge 0$. Moreover, the infimum can be taken over $C^1_{\rm comp}(M)$.
	\subsection{Bottom of the spectrum on finite covers} \label{subsec: bottom spec cover} It turns out that the bottom of the (essential) spectrum of the Laplacian is preserved under finite coverings.
	\begin{lem}\label{lem: bottom spec cover} Let $p:M_1\to M_0$ be a Riemannian covering of finite degree. Then,
		\[\lambda_0(M_1)=\lambda_0(M_0), \qquad \lambda_{\rm ess}(M_1)=\lambda_{\rm ess}(M_0).\]
	\end{lem}
	This result is known (see e.g. \cite{ballmann2018bottom} and \cite[\S 2.1]{BallmannInstability2025}) but for the convenience of the reader we recall the proof.
	\begin{proof} 1. Let $f\in C^\infty_{\rm comp}(M_0)$. Then, the pullback $f\circ p$ has compact support in $M_1$ and
		\[\operatorname{Ray}(f,M_0)=\operatorname{Ray}(f\circ p,M_1).\]
		Taking the infimum over $f$ yields $\lambda_0(M_0)\ge \lambda_0(M_1)$.\medskip
		
		2. Let $g\in C^1_{\rm comp}(M_1)$ be nonnegative, and define $f\in C^1_{\rm comp}(M_0)$ by
		\[f(x):=\Big(\sum_{y\in p^{-1}(x)} g(y)^2\Big)^{\frac 12}.\]
		Since $g\ge 0$, one can check that $f\in C^1_{\rm comp}(M_0)$. On the one hand
		\[\int_{M_0}|f(x)|^2\mathrm{d}x=\int_{M_0} \sum_{y\in p^{-1}(x)} g(y)^2\mathrm{d}x=\int_{M_1} g(y)^2\mathrm{d}y.\]
		On the other hand, one has
		\[2f\nabla f(x)=\nabla f^2(x)=\sum_{y\in p^{-1}(x)} \nabla g^2(y)=\sum_{y\in p^{-1}(x)} 2g\nabla g(y),\]
		which gives
		\[\int_{M_0} |\nabla f(x)|^2\mathrm{d}x=\int_{M_0} \frac{1}{f^2}\Big|\sum_{y\in p^{-1}(x)} g\nabla g(y)\Big|^2\mathrm{d}x.\]
		By Cauchy--Schwarz, the integrand is bounded by
		\[ \frac{1}{f^2}\Big(\sum_{y\in p^{-1}(x)} g(y)^2\Big)\cdot \Big(\sum_{y\in p^{-1}(x)} |\nabla g(y)|^2\Big)= \sum_{y\in p^{-1}(x)} |\nabla g(y)|^2. \]
		Integrating over $M_0$ leads to
		\[\int_{M_0} |\nabla f(x)|^2\mathrm{d}x\le \int_{M_1} |\nabla g(y)|^2\mathrm{d}y.\]
		Altogether,
		\[\operatorname{Ray}(f,M_0)\le \operatorname{Ray}(g,M_1).\]
		Taking the infimum over $g$ leads to the reverse inequality $\lambda_0(M_0)\le \lambda_0(M_1)$.\medskip
		
		3. Let us now turn to the equality $\lambda_{\rm ess}(M_1)=\lambda_{\rm ess}(M_0)$. We recall the decomposition principle of Donnelly--Li \cite{DonnellyLi79}: if $\{K_n\}$ is an exhaustion of $M_1$ by compact sets then
		\[\lambda_{\rm ess}(M_1)=\lim_{n\to +\infty} \lambda_0(M_1\backslash K_n).\]
		But then $\{p(K_n)\}$ is an exhaustion of $M_0$, thus
		\[\lambda_{\rm ess}(M_0)=\lim_{n\to +\infty} \lambda_0(M_0\backslash p(K_n)).\]
		Now, for each $n$, $M_1\backslash K_n$ is a finite degree cover of $M_0\backslash p(K_n)$, thus for all $n$, by the first point of the Lemma,
		\[\lambda_0(M_1\backslash K_n)=\lambda_0(M_0\backslash p(K_n)).\]
		Letting $n\to +\infty$ gives the sought equality.
	\end{proof}
	
	\subsection{The heat kernel} Since the spectrum of $\Delta_M$ is a subset of $\mathbf{R}_+$, one can define the \emph{heat operator} $\mathrm{e}^{-t\Delta_M}:L^2(M)\to L^2(M)$ by functional calculus, for $t\ge 0$. There exists a unique smooth function $H_M(t,x,y)\in C^\infty(\mathbf{R}_+^*\times M\times M)$, called the \emph{heat kernel}, such that for all $f\in L^2(M)$,
	\begin{equation} \label{eq: def heat kernel} \mathrm{e}^{-t\Delta_M} f(x)=\int_M H_{M}(t,x,y)f(y)\mathrm{d}y, \qquad \forall t>0,\forall x\in M.\end{equation}
	Notably, $H_M(t,x,y)$ is the Schwartz kernel of the heat operator $\mathrm{e}^{-t\Delta_M}$. The heat kernel is nonnegative, symmetric ($H_M(t,x,y)=H_M(t,y,x)$), and satisfies
	\[\int_M H_M(t,x,y)\mathrm{d}y\le 1,\qquad \text{for all $t>0$ and $x\in M$}.\]
	Formula \eqref{eq: def heat kernel} allows one to define $\mathrm{e}^{-t\Delta_M} f$ for smooth bounded functions $f\in C^\infty_{\rm b}(M)$ by setting
	\[\forall t>0, \ \mathrm{e}^{-t\Delta_M} f(x):=\int_M H_M(t,x,y)f(y)\mathrm{d}y.\]
	Then, $\mathrm{e}^{-t\Delta_M}f(x)$ is a smooth solution to the Cauchy problem
	\begin{equation} \label{eq: cauchy problem}\left\{\begin{array}{ll}
		\partial_t u=-\Delta_M u & \text{in $\mathbf{R}_+^*\times M$,} \\
		u(0,x)=f(x).
	\end{array}\right.\end{equation}
The initial condition is understood in the sense that $u(t,x){\longrightarrow} f(x)$ as $t\to 0$, uniformly on compact sets. In general, solutions to \eqref{eq: cauchy problem} need not be unique. However, it is guaranteed under suitable geometric assumptions.
	\begin{propo}\label{prop: uniqueness cauchy} Assume that $(M,g)$ is complete, with Ricci curvature bounded from below. Then, for every $f\in C^\infty_{\rm b}(M)$, the Cauchy problem \eqref{eq: cauchy problem} has a unique bounded solution $u(t,x)\in C^\infty(\mathbf{R}_+^*\times M)$ given by $\mathrm{e}^{-t\Delta_M}f$. \end{propo}
	
	\subsection{Heat kernel upper bounds} We recall useful estimates on solutions to the heat equation. The first is a \emph{Parabolic Harnack inequality}, which originates from the work of Li and Yau \cite{Li1986}.
	\begin{lem}[Parabolic Harnack inequality, see {\cite[Theorem 5.2.5]{Davies_1989}}] Let $(M,g)$ be a complete Riemannian surface with sectional curvature bounded by $-b^2$ from below. Then, for any nonnegative solution $u(t,x)$ to the heat equation, for all $x,y\in M$ and $t,s\ge 0$,
	\[u(t,x)\le u(t+s,y)\Big(\frac{t+s}{t}\Big)^2\exp\Big(\frac{d(x,y)^2}{2s}+b^2s\Big).\] \end{lem}

Next, we recall the so-called \emph{Takeda's inequality} from \cite{Grigor’yan_1994}, in a weakened form that is sufficient for our purposes. 
\begin{lem}[Takeda's inequality]Let $(M,g)$ be a complete noncompact Riemannian manifold. For $r>0$, let $V(x,r)$ denote the volume of the ball $B(x,r)$. Then, there are absolute constants $C,c>0$ such for any $x\in M$ and $R\ge 2r+1$, one has
	\begin{multline}\int_{B(x,r)\times M\backslash B(x,R)}H_M(t,x,y)\mathrm{d}x\mathrm{d}y \le C \sqrt{V(x,R)V(x,r)} (1+\sqrt t)\exp(-cR^2/t).\end{multline}\end{lem}

With these two Lemmas, we can estimate the mass of the heat kernel outside large balls.
\begin{propo} \label{propo: mass heat large balls} Let $(M,g)$ be a complete Riemannian manifold with sectional curvature $-b^2\le K_g\le -1$. Let $x\in M$, and let $\rho:=\min(1,\mathrm{inj}(x))$. There are constants $R_0,C,c>0$ such that if $R\ge R_0$ then
	\[\int_{M\backslash B(x,R)} H_M(t,x,y)\mathrm{d}y\le C \frac{1+\sqrt t}{\rho} \exp(CR-cR^2/t).\]
\end{propo}
\begin{proof}Let $s=\min(1,t)$. Then, for any $z\in B(x,\rho)$ and $y\in X\backslash B(x,R)$, one has by the parabolic Harnack inequality:
	\[H_M(t,x,y)\le C H_M(t+s,z,y)\exp\Big(\frac{1}{2s}\Big).\]
	It follows that
	\[\int_{M\backslash B(x,R)}H_M(t,x,y)\mathrm{d}y\le  C\frac{\mathrm{e}^{\frac{1}{2s}}}{V(x,\rho)}\int_{B(x,\rho)\times M\backslash B(x,R)}H_M(t+s,z,y)\mathrm{d}z\mathrm{d}y.\]
	We now apply Takeda's inequality to get
	\begin{multline}\int_{M\backslash B(x,R)}H_M(t,x,y)\mathrm{d}y \le C \frac{\sqrt{V(x,R)}}{\sqrt{V(x,\rho)}} (1+\sqrt t)\exp\big(\frac{1}{2s}-c\frac{R^2}{(t+s)}\big).\end{multline}
	Now, if $R$ is large enough, one has
	\[\exp\big(\frac{1}{2s}-c\frac{R^2}{(t+s)}\big)\le \exp(-cR^2/4t).\]
	Using the estimates $V(x,R)\le C\mathrm{e}^{bR}$ and $V(x,\rho)\ge  \rho^2$, we infer (up to decreasing the constant $c$)
\[\int_{M\backslash B(x,R)}H_M(t,x,y)\mathrm{d}y\le C \frac{1+\sqrt t}{\rho} \exp(bR/2-cR^2/t),\]
as we wished.
	\end{proof}

\subsection{Heat operator in the ends} \label{sec: heat end} Let $X=\Gamma\backslash \widetilde X$ be a complete geometrically finite surface with pinched negative curvature $-b^2\le K_g\le -1$. Recall that if $\phi$ is a homomorphism $\Gamma\to \mathbf S_n$, then $X_\phi$ denotes the associated Riemannian covering of $X$. In this section, we control the behavior of the heat operator $\mathrm{e}^{-t\Delta_{X_\phi}}$ near the ends of $X_\phi$. Fix a point $o$ on the base surface $X$ and define
\[\mathcal C_{R}:=\{x\in X, \ d(x,o)> R\}.\]
Since $X$ is geometrically finite, for $R$ large enough, $\mathcal C_{R}$ consists of a disjoint union of a finite number of ends. We let $\mathcal C_R^\phi$ denote the lift of $\mathcal C_R$ to $X_\phi$, and let $\chi_{\mathcal C_R^\phi}$ denote the indicator function of the set $\mathcal C_R^\phi$. \begin{propo}\label{prop: heat ends} Fix some radii $R,r>0$ and assume that $R\ge 2r$ and $R\ge t^2+R_0$ with $R_0$ large enough. Consider the operator
	\[ \exp(-t\Delta_{X_\phi})\chi_{\mathcal C_R^\phi}:L^2(X_\phi)\to L^2(X_\phi).\]
	Then,
	\begin{equation} \label{eq: heat in the ends} \exp(-t\Delta_{X_\phi})\chi_{\mathcal C_R^\phi}=\chi_{\mathcal C_r^\phi}\exp(-t\Delta_{{\mathcal C_r^\phi}})\chi_{\mathcal C_R^\phi}+\mathcal R,\end{equation}
	where $\Delta_{\mathcal C_r^\phi}$ is the Dirichlet Laplacian in $\mathcal C_r^\phi$ and the error is controlled by \[\|\mathcal R\|_{L^2(X_\phi)}\le \exp(-t^2).\] Note that the error is uniform with respect to the covering $X_\phi\to X$. \end{propo}
 \begin{rem}We make two comments before the proof:
 	\begin{itemize}[leftmargin=1em]
  \item Although the heat equation has infinite speed of propagation, the operator $\chi_{\mathcal C_r^\phi}\exp(-t\Delta_{{\mathcal C_r^\phi}})\chi_{\mathcal C_R^\phi}$ appearing on the right-hand side of \eqref{eq: heat in the ends} depends only on the geometry of $X$ in $\mathcal C_r$ and does not see the rest of the manifold. By the results of \S \ref{subsec: bottom spec cover}, the bottom of the spectrum of $\Delta_{\mathcal C_r^\phi}$ converges to $\lambda_{\rm ess}(X)$ as $r\to +\infty$. This will allow us to control the operator norm of $\exp(-t\Delta_{X_\phi})\chi_{\mathcal C_R^\phi}$ in terms of the bottom of the essential spectrum $\lambda_{\rm ess}(X)$.
  \item Proposition \ref{prop: heat ends} gives a quantitative formulation of the following probabilistic heuristic: for times $t\ll R$, with high probability, a Brownian motion starting in $\mathcal C_R$ remains in $\mathcal C_r$ for all times $\le t$.
\end{itemize} \end{rem} In the following we drop the $\phi$ superscripts to ease notations.
	\begin{proof} Denote $\mathcal R=(\mathrm{e}^{-t\Delta_{X_\phi}}-\chi_{\mathcal C_r}\mathrm{e}^{-t\Delta_{\mathcal C_r}})\chi_{\mathcal C_R}$. Then, for $f\in L^2(X_\phi)$, one has
	\[\mathcal R f(x)=\int_{X_\phi} \big(H_{X_\phi}(t,x,y)-\chi_{\mathcal C_r}(x)H_{\mathcal C_r}(t,x,y)\big)\chi_{\mathcal C_R}(y)f(y)\mathrm dy.\]
	Here $H_{\mathcal C_r}$ denotes the Dirichlet heat kernel in $\mathcal C_r$.
	We apply Schur's Lemma to bound the operator norm of $\mathcal R$. Let
	\[C_1:=\sup_{x\in X_\phi} \int_{X_\phi}\big(H_{X_\phi}(t,x,y)-\chi_{\mathcal C_r}(x)H_{\mathcal C_r}(t,x,y)\big) \chi_{\mathcal C_R} (y) \mathrm dy,\]
	\[C_2:=\sup_{y\in X_\phi}  \chi_{\mathcal C_R} (y) \int_{X_\phi}\big(H_{X_\phi}(t,x,y)-\chi_{\mathcal C_r}(x)H_{\mathcal C_r}(t,x,y)\big) \mathrm dx.\]
	Then $\|\mathcal R\|\le \sqrt{C_1C_2}$.\medskip
	
	We start by estimating $C_1$. 
	
	1. We first assume $x\in X_\phi\backslash \mathcal C_r$. In this case, $\chi_{\mathcal C_r}(x)=0$ thus 
	\[\int_{X_\phi}\big(H_{X_\phi}(t,x,y)-\chi_{\mathcal C_r}(x)H_{\mathcal C_r}(t,x,y)\big) \chi_{\mathcal C_R} (y) \mathrm dy=\int_{\mathcal C_R}H_{X_\phi}(t,x,y) \mathrm dy.\]
	Since $x\in X_\phi\backslash \mathcal C_r$, one has $B(x,R-r)\cap \mathcal C_R=\varnothing$. Since $R\ge 2r$, letting $\rho(x)=\min(1,\operatorname{inj}(x))$, we get from Proposition \ref{propo: mass heat large balls} that 
	\[\int_{\mathcal C_R}H_{X_\phi}(t,x,y) \mathrm dy\le C\frac{1+\sqrt t}{\rho(x)}\exp(CR-cR^2/t).\]
	Lemma \ref{lem: injrad} implies $\rho(x)\ge C\exp(-br)\ge C\exp(-bR)$, thus up to increasing the constant $C$ in the exponential, we have
	\[\int_{\mathcal C_R}H_{X_\phi}(t,x,y) \mathrm dy\le C(1+\sqrt t)\exp(CR-cR^2/t).\]
	\medskip 
	
2. We now assume $x\in \mathcal C_r$. In this case, one has
	\[\int_{X_\phi}\big(H_{X_\phi}(t,x,y)-\chi_{\mathcal C_r}(x)H_{\mathcal C_r}(t,x,y)\big) \chi_{\mathcal C_R} (y) \mathrm dy=\int_{X_\phi}\big(H_{X_\phi}(t,x,y)-H_{\mathcal C_r}(t,x,y)\big) \chi_{\mathcal C_R} (y) \mathrm dy.\]
	The right-hand side defines a smooth function $u(t,x)\ge 0$ for $t>0,x\in \mathcal C_r$. Moreover :
	\begin{itemize}
		\item $\partial_t u=-\Delta u$ in $\mathbf R_+^*\times \mathcal C_r$,
		\item $u(t,x)\to 0$ when $t\to 0$ locally uniformly in $\mathcal C_r$,
		\item $u(t,x)=\int_{\mathcal C_R} H_{X_\phi}(t,x,y)\mathrm dy$ on $\mathbf R_+^* \times \partial \mathcal C_r$,
		\item For all $T>0$, $u(t,x)$ goes to $0$ when $x\to \infty$ in $X_\phi$, uniformly in $t\in (0,T)$. This last fact will we proved below.
	\end{itemize}
	By the parabolic maximum principle (see e.g. \cite[Corollary 5.20 and Theorem 8.11]{grigoryan2009heat}), we get that for all $x\in \mathcal C_r$:
	\[u(t,x)\le \sup_{(\tau,x)\in [0,t]\times \partial \mathcal C_r} u(\tau,x)= \sup_{(\tau,x)\in [0,t]\times \partial \mathcal C_r}  \int_{\mathcal C_R} H_{X_\phi}(\tau,x,y)\mathrm dy.\]
	Since $R\ge 2r$, for any $x\in \partial \mathcal C_r$, the ball $B(x,R/2)$ is disjoint from $\mathcal C_R$, thus by Proposition \ref{propo: mass heat large balls}:
	\[u(\tau,x) \le\int_{X_\phi\backslash B(x,R/2)} H_{X_\phi}(\tau,x,y)\mathrm dy\le  C \frac{1+\sqrt \tau}{\rho(x)}\exp(CR-cR^2/\tau).\]
	For $x\in \partial \mathcal C_r$, we have by Lemma \ref{lem: injrad}:
	\[\rho(x)\ge C\exp(-br)\ge C\exp(-bR).\]
	Therefore, up to increasing the constant $C$ in the exponential, we have for any $x\in \partial \mathcal C_r$ and $\tau\in (0,t)$:
	\[u(\tau,x)\le C (1+\sqrt t)\exp(CR-cR^2/t).\]
	
	To estimate $C_2$, we just use the fact that $0\le H_{\mathcal C_r}\le H_{X_{\phi}}$ and that the heat kernel has mass $1$ to obtain $C_2\le 1$. By Schur's Lemma, we obtain
	\[\big\|\mathcal R\big\|\le C (1+\sqrt t)\exp(CR-cR^2/t)\]
	for other constants $C,c>0$. If $R\ge R_0+t^2$ with $R_0$ large enough, the error can be made $\le \exp(-t^2)$.
	\end{proof}
	
	We finally prove the claim that $u(t,x)\to 0$ when $x\to \infty$ in $X_\phi$ uniformly in $t\in (0,T)$. This means that for any exhaustion $\{K_n\}$ of $X_\phi$ by compact subsets, one has
	\[ \sup_{t\in (0,T)} \sup_{x\in X_\phi\backslash K_n} u(t,x)\underset{n\to +\infty}{\longrightarrow} 0.\]
	\begin{lem}For $x\in \mathcal C_r^\phi$, let 
	\[u(t,x)=\int_{X_\phi}\big(H_{X_\phi}(t,x,y)-H_{\mathcal C_r^\phi}(t,x,y)\big) \chi_{\mathcal C_R^\phi} (y) \mathrm dy.\]
	Then for all $T>0$, $u(t,x)\to 0$ as $x\to \infty$ in $X_\phi$, uniformly in $t\in (0,T)$. 
	\end{lem}
	We prove the result on $X$ to ease notations, but the proof extends seamlessly to finite covers. 
	\begin{proof} Since $X$ is connected, the statement to prove is equivalent to
		\[\underset{d(x,o)\to +\infty}{\lim} \big(\sup_{t\in (0,T)}  u(t,x)\big)=0,\]
		where $o$ is a fixed point in $X$.
		
	We first observe that since the heat kernel has mass $\le 1$, we have
	\[u(t,x)\le 1-\int_{X}H_{\mathcal C_r}(t,x,y) \chi_{\mathcal C_R} (y) \mathrm dy.\]
	Let $x\in \mathcal C_R$ be at distance $2r(x)$ from the boundary $\partial \mathcal C_R$. Then, let $\omega\in C^\infty(X,[0,1])$ be a smooth cutoff as constructed in Lemma \ref{lem: cutoff}, with the following properties:
	\begin{itemize}
  \item $\omega\equiv 1$ in $B(x,r(x))$,
  \item $\operatorname{Supp}(\omega)\subset B(x,2r(x))\subset \mathcal C_R$.
  \item letting $\rho(x):=\min(\operatorname{inj}(x),1)$, we have
  \begin{equation} \label{eq: upper bound Laplacian omega} |\Delta \omega|\le C\rho(x)^{-4}\exp(C r(x)).\end{equation}
\end{itemize}
	Then, we have
	\[0\le u(t,x)\le 1-\int_{X}H_{\mathcal C_r}(t,x,y) \omega (y) \mathrm dy.\]
	We denote the right-hand side by $v_x(t)$. Since $\omega(x)=1$, we have $v_x(0)=0$. We now compute the time derivative $v_x'(t)$. The integrand is smooth and compactly supported, so we can differentiate under the integral. Using the heat equation $\partial_t H_{\mathcal C_r}=-\Delta_y H_{\mathcal C_r}$, and integrating by parts in the $y$ variable, we find
	\[v_x'(t)=\int_{X} H_{\mathcal C_r}(t,x,y)\Delta \omega (y)\mathrm{d}y.\]
	Since $0\le H_{\mathcal C_r}\le H_X$, it follows from the triangle inequality that
	\[|v_x'(t)|\le \int_{X} H_X(t,x,y)|\Delta \omega(y)|\mathrm{d}y.\]
	Because $\Delta \omega\equiv 0$ in $B(x,r(x))$, and using the upper bound \eqref{eq: upper bound Laplacian omega}, we get
	\[|v_x'(t)|\le C\rho(x)^{-4}\exp(Cr(x)) \int_{X\backslash B(x,r(x))} H_X(t,x,y)\mathrm{d}y.\]
	By Proposition \ref{propo: mass heat large balls}, we find that for all $t\in(0,T)$, up to increasing the constant $C$ in the exponential, one has
	\[|v_x'(t)|\le C(T)\rho(x)^{-5}\exp(Cr(x)-cr(x)^2/T).\]
	By Lemma \ref{lem: injrad} again, we have a lower bound
	\[\rho(x)\ge C \exp(-bd(x,o)),\]
	and since $|r(x)-d(x,o)|$ is bounded by a constant depending only on $R$, we conclude that for all $t\in(0,T)$,
	\[|v_x'(t)|\le C(R,T)\exp\Big(Cr(x)-cr(x)^2/T\Big),\]
	for some other constant $C$. Provided $r(x)$ is large enough (and $T$ remains bounded) and replacing $c$ by $c/2$, we infer that for all $t\in(0,T)$,
	\begin{equation} \label{eq: upper bound derivative v_x}|v_x'(t)|\le C(R,T)\exp(-c r(x)^2/T).\end{equation}
	Integrating \eqref{eq: upper bound derivative v_x} between $0$ and $t$, it follows that for $r(x)$ large enough depending on $T$, one has
	\[|v_x(t)|\le C(R,T)\exp(-cr(x)^2/T).\]
	This quantity goes to $0$ as $r(x)\to +\infty$ (or equivalently $d(x,o)\to +\infty$), uniformly in $t\in (0,T)$.
	\end{proof}

\section{Proof of Theorem \ref{thm: main}} \label{sec: proof main}

\subsection{Strong convergence}\label{Sec-strong} A sequence of representations $(\rho_n,V_n)$ of a discrete group $\Gamma$ is said to converge strongly if for all $z\in \mathbf C[\Gamma]$, one has
\[\lim_{n\to +\infty} \|\rho_n(z)\|_{V_n}=\|\rho_\infty(z)\|_{\ell^2(\Gamma)},\]
where $(\rho_\infty,\ell^2(\Gamma))$ denotes the regular representation of $\Gamma$.\medskip 

Let $V_n^0=\ell_0^2([n])$ denote the $(n-1)$-dimensional space of square summable functions with zero mean over $[n]$. Recall that if $\phi\in \mathrm{Hom}(\Gamma,\mathbf S_n)$ is a random homomorphism, we let $\rho_\phi=\operatorname{std}_{n-1}\circ \phi$, where $(\operatorname{std}_{n-1},V_{n}^0)$ is the standard representation of $\mathbf S_n$. Bordenave--Collins \cite{bordenave2019eigenvalues} showed that for any $\varepsilon>0$, for all $z\in \mathbf C[\Gamma]$, one has with probability tending to $1$ as $n\to +\infty$:
\[\|\rho_\phi(z)\|_{V_n^0}\le \|\rho_\infty(z)\|_{\ell^2(\Gamma)}+\varepsilon.\]
By matrix amplification, see e.g. \cite[Proposition 3.3]{Magee_survey}, and approximation on both sides by finite rank operators, one obtains:
\begin{thm}[Bordenave--Collins]\label{prop: strong conv} Let $\mathcal K(L^2(\mathcal F))$ denote the space of compact operators on $L^2(\mathcal F)$. Let $\varepsilon>0$. Then, for any finitely supported map $\gamma\in\Gamma\mapsto a_\gamma\in \mathcal K(L^2(\mathcal F))$, one has
	\begin{equation}\label{eq: strong conv} \Big\|\sum_{\gamma\in\Gamma} a_\gamma\otimes \rho_\phi(\gamma)\Big\|_{L^2(\mathcal F)\otimes V_n^0}\le \Big\|\sum_{\gamma\in\Gamma} a_\gamma\otimes \rho_\infty(\gamma)\Big\|_{L^2(\mathcal F)\otimes \ell^2(\Gamma)}+\varepsilon,\end{equation}with probability tending to $1$ as $n\to +\infty$.\end{thm}
The hope is then that the limit operator appearing on the right-hand side of \eqref{eq: strong conv} is easier to understand.

\subsection{Splitting the operator} \label{subsec: splitting}
Let $X_\phi\to X$ be a finite degree cover associated with $\phi$. By Lemma \ref{lem: bottom spec cover}, the bottom of the essential spectrum is preserved under such covers: 
\[\lambda_{\rm ess}(X_\phi)=\lambda_{\rm ess}(X).\] 
Moreover, since $X$ is geometrically finite, \cite[Theorem A]{BallmannEssentialBottom} ensures that $\lambda_{\rm ess}(X)\ge \lambda_0(\widetilde X)$. It follows that the spectrum of $X_\phi$ in the interval $[0,\lambda_0(\widetilde X))$ consists of a discrete set of eigenvalues of finite multiplicities.\medskip

Let $L^2_{\rm new}(X_\phi)\subset L^2(X_\phi)$ denote the subspace of functions that are orthogonal to all lifts of functions from $X$. Then, one has a natural orthogonal decomposition
\begin{equation}\label{eq: splitting L^2} L^2(X_\phi)\simeq L^2_{\rm new}(X_\phi)\oplus L^2(X).\end{equation}
Moreover, the Laplacian acts diagonally on this direct sum and we have
\[\sigma(\Delta_{X_\phi})\cap [0,\lambda_0(\widetilde X))=\Big(\sigma\big(\Delta_{X_\phi}|_{L^2_{\rm new}(X_\phi)}\big)\cup \sigma(\Delta_X)\Big)\cap [0,\lambda_{0}(\widetilde X)),\]
where eigenvalues are counted with multiplicities on both sides. To prove Theorem \ref{thm: main}, it is enough to show that with high probability as $n\to +\infty$, one has
\begin{equation}\label{eq: no spec below lambda_0} \sigma\big(\Delta_{X_\phi}|_{L^2_{\rm new}(X_\phi)}\big)\cap \big[0,\lambda_{0}(\widetilde X)-\varepsilon\big]=\varnothing.\end{equation}
By the spectral theorem, for any $t>0$, \eqref{eq: no spec below lambda_0} is equivalent to
\[\big\|\exp(-t \Delta_{X_\phi})\big\|_{L^2_{\rm new}(X_\phi)}<\mathrm{e}^{-t(\lambda_0(\widetilde X)-\varepsilon)}.\]
\subsection{Localizing away from the ends} \label{subsec: localize away ends} Before proceeding, we localize the heat operator away from the ends. Let $(R,t,r)$ be as in Proposition \ref{prop: heat ends}, and write
\[\exp(-t \Delta_{X_\phi})=\exp(-t \Delta_{X_\phi})\chi_{\mathcal C_R^\phi}+\exp(-t \Delta_{X_\phi})(1-\chi_{\mathcal C_R^\phi}).\]
By Proposition \ref{prop: heat ends}, one has
\[\exp(-t \Delta_{X_\phi})\chi_{\mathcal C_R^\phi}=\chi_{\mathcal C_r^\phi}\exp(-t \Delta_{\mathcal C_r^\phi})\chi_{\mathcal C_R^\phi}+\mathcal R,\]
with $\|\mathcal R\|\le \mathrm{e}^{-t^2}$. Moreover, by the spectral theorem, one has
\[\big \|\exp(-t \Delta_{\mathcal C_r^\phi})\big \|=\mathrm{e}^{-t\lambda_0(\mathcal C_r^\phi)}.\]
Recall from Lemma \ref{lem: bottom spec cover} that $\lambda_0(\mathcal C_r^\phi)=\lambda_0(\mathcal C_r)$. Taking $R=t^2$ large enough above (and depending on $r$) to arrange that the error is $\le \mathrm{e}^{-t\lambda_0(\widetilde X)}$, we arrive at
\begin{equation}\label{eq: bound op norm after splitting} \big\|\exp(t\Delta_{X_\phi})\big \|_{L^2_{\rm new}(X_\phi)}\le \big\|\exp(-t\Delta_{X_\phi})(1-\chi_{\mathcal C_R^\phi})\big\|_{L^2_{\rm new}(X_\phi)}+\mathrm{e}^{-t\lambda_0(\mathcal C_r)}+\mathrm{e}^{-t\lambda_0(\widetilde X)}.\end{equation}
Thus, we are left with estimating $\big\|\exp(-t\Delta_{X_\phi})(1-\chi_{\mathcal C_R^\phi})\big\|_{L^2_{\rm new}(X_\phi)}$ with high probability as $n\to +\infty$.

\subsection{Lifting to the universal cover} Let $C^\infty_{\rm new}(X_\phi)$ denote the space of smooth functions on $X_\phi$ that are orthogonal to all lifts of functions from $X$. Let $C^\infty_\phi(\widetilde X,V_n^0)$ denote the space of smooth $V_n^0$-valued functions satisfying the equivariance property
\[f(\gamma x)=\rho_\phi(\gamma)f(x).\]
Let $\mathcal F$ denote a Dirichlet fundamental domain for the action of $\Gamma$ on $\widetilde X$ as described in \S \ref{subsec: dirichlet domain}, and let $L^2_\phi(\widetilde X,V_n^0)$ denote the closure of the space $C^\infty_\phi(\widetilde X,V_n^0)$ with respect to the norm
\[\|f\|_{L^2_\phi(\widetilde X,V_n^0)}:=\int_{\mathcal F} \|f(x)\|_{V_n^0}\mathrm{d}x.\]
Then, we have an isomorphism $C^\infty_{\rm new}(X_\phi)\simeq C^\infty_{\phi}(\widetilde X,V_n^0)$, that extends to an isomorphism of Hilbert spaces
\[U:L^2_{\rm new}(X_\phi)\simeq L^2_\phi(\widetilde X,V_n^0).\]
Moreover, the operator $U$ intertwines $\Delta_{X_\phi}$ and $\Delta_{\widetilde X}\otimes \mathrm{Id}_{V_n^0}$. By the uniqueness of bounded solutions to the Cauchy problem, we get that for any $f\in L^2_\phi(\widetilde X,V_n^0)\cap C^\infty_{\rm b}(\widetilde X,V_n^0)$, one has
\[(U\mathrm{e}^{-t\Delta_{X_\phi}}U^{-1})f=(\mathrm{e}^{-t\Delta_{\widetilde X}}\otimes \mathrm{Id}_{V_n^0}) f.\]
Now, for $f\in L^2_\phi(\widetilde X,V_n^0)\cap C^\infty_{\rm b}(\widetilde X,V_n^0)$, one has
\[\big(\mathrm{e}^{-t\Delta_{\widetilde X}}\otimes \mathrm{Id}_{V_n^0}\big) f(x)=\int_{\widetilde X} H_{\widetilde X}(t,x,y)f(y)\mathrm{d}y.\]
Using Proposition \ref{prop: fundamental domain}, then performing a change of variables $y\mapsto \gamma y$ in each of the integrals and using the relations $f(\gamma y)=\rho_\phi(\gamma)f(y)$, one gets by Fubini:
\[\big(U \mathrm{e}^{-t\Delta_{X_\phi}} U^{-1}\big)f(x)=\int_{\mathcal F}\Big(\sum_{\gamma\in \Gamma} H_{\widetilde X}(t,x,\gamma y)\rho_\phi(\gamma)\Big)f(y)\mathrm{d}y.\]
This identity extends by continuity to all $f\in L^2_\phi(\widetilde X,V_n^0)$.

Following \cite[\S 6.1]{HM}, we have another isomorphism of Hilbert spaces
\[\begin{array}{rl}L^2_\phi(\widetilde X,V_n^0)& \simeq L^2(\mathcal F)\otimes V_n^0 \\ f& \mapsto \displaystyle \sum_j \langle f_{|\mathcal F},e_j\rangle_{V_n^0}\otimes e_j\end{array},\]
where $\{e_j\}$ is an arbitrary orthonormal basis of $V_n^0$. Under these two identifications, the operator $\mathrm{e}^{-t\Delta_{X_\phi}}(1-\chi_{\mathcal C_R^\phi})$ is conjugated to
\begin{equation}\label{eq: lift heat operator} \mathrm{e}^{-t\Delta_{X_\phi}}(1-\chi_{\mathcal C_R^\phi})\simeq\sum_{\gamma\in \Gamma} a_\gamma \otimes \rho_\phi(\gamma),\end{equation}
where $a_\gamma:L^2(\mathcal F)\to L^2(\mathcal F)$ is an integral operator with kernel
\[a_\gamma(x,y)=H_{\widetilde X}(t,x,\gamma y)(1-\chi_{\widetilde{\mathcal C}_R})(y).\]
Here $\widetilde{\mathcal C}_R=\pi^{-1}(\mathcal C_R)$ where $\pi:\widetilde X\to X$ is the universal cover. We can still not apply the strong convergence machinery because the sum on the right-hand side of \eqref{eq: lift heat operator} contains an infinite number of operators. First, we perform another truncation by localizing the kernel in the set $\{d(x,\gamma y)\le R\}$. Let $a_\gamma^{(1)}$ (resp. $a_\gamma^{(2)}$) be the integral operator with kernel $a_\gamma(x,y)\mathbf 1_{d(x,\gamma y)\le R}$ (resp. $a_\gamma(x,y)\mathbf 1_{d(x,\gamma y)> R}$), so that $a_\gamma=a_\gamma^{(1)}+a_\gamma^{(2)}$.

\subsection{Dealing with $a^{(2)}$} We first bound the operator norm of 
\[\sum_{\gamma\in \Gamma} a_\gamma^{(2)} \otimes \rho_\phi(\gamma)\]
 acting on $L^2(\mathcal F)\otimes V_n^0$. Note that this operator is just the restriction to $L^2(\mathcal F)\otimes V_n^0$ of the operator
\[A:=\sum_{\gamma\in \Gamma} a_\gamma^{(2)} \otimes \sigma_\phi(\gamma)\]
acting on $L^2(\mathcal F)\otimes \ell^2([n])$, where $\sigma_\phi$ denotes the composition of $\phi$ with the permutation representation of $\mathbf S_n$:
\[\Gamma \overset{\phi}{\longrightarrow}\mathbf S_n\overset{\mathrm{perm}_n}{\longrightarrow} \mathcal U(\ell^2([n])).\]
The operator norm of $A$ will be controlled using Schur's Lemma. Under the identification $L^2(\mathcal F)\otimes \ell^2([n])\simeq L^2(\mathcal F\times [n])$, the operator $A$ has kernel
\[A((x,i),(y,j))=\sum_{\gamma\in \Gamma}\mathbf 1_{d(x,\gamma y)> R}H_{\widetilde X}(t,x,\gamma y)(1-\chi_{\widetilde{\mathcal C}_R})(y)\delta_{i,\phi[\gamma](j)}.\]
Note that this kernel is nonnegative. Let
\[C_1=\sup_{(y,j)\in \mathcal F\times [n]} \int_{\mathcal F\times [n]} A((x,i),(y,j))\mathrm{d}x\mathrm{d}i,\]
\[C_2=\sup_{(x,i)\in \mathcal F\times [n]}\int_{\mathcal F\times [n]}A((x,i),(y,j))\mathrm{d}y\mathrm{d}j.\]
Here $\mathrm{d}i$ is just the counting measure on $[n]$. For any $\gamma\in \Gamma$, we have $\sum_{i=1}^n \delta_{i,\phi[\gamma](j)}=1$. This leads to
\[C_1=\sup_{(y,j)} \sum_{\gamma\in \Gamma} \int_{\mathcal F} \mathbf 1_{d(x,\gamma y)> R}H_{\widetilde X}(t,x,\gamma y)(1-\chi_{\widetilde{\mathcal C}_R})(y)\mathrm{d}x.\]
Since the heat kernel is invariant under isometries, we have
\[\mathbf 1_{d(x,\gamma y)> R}H_{\widetilde X}(t,x,\gamma y)=\mathbf 1_{d(\gamma^{-1}x,y)> R}H_{\widetilde X}(t,\gamma^{-1}x,y).\]
Using Proposition \ref{prop: fundamental domain}, one finds
\[C_1=\sup_{y} \int_{\widetilde X} \mathbf 1_{d(x,y)> R}H_{\widetilde X}(t,x,y)(1-\chi_{\widetilde{\mathcal C}_R})(y)\mathrm{d}x.\]
In particular,
\[C_1\le \sup_{y} \int_{\widetilde X} \mathbf 1_{d(x,y)> R}H_{\widetilde X}(t,x,y)\mathrm{d}x.\]
Since $\widetilde X$ has infinite injectivity radius, Proposition \ref{propo: mass heat large balls} yields
\[C_1\le C\exp(CR-cR^2/t)\]
Similarly, we obtain $C_2\le C\exp(CR-cR^2/t)$. By taking $R\ge R_0+t^2$ large enough (depending on $r$), we get
\begin{equation} \label{eq: bound a^(2)}\Big\|\sum_{\gamma\in \Gamma} a_\gamma^{(2)} \otimes \rho_\phi(\gamma)\Big\|_{L^2(\mathcal F)\otimes V_n^0}\le \|A\|_{L^2(\mathcal F\times [n])} \le  \mathrm{e}^{-\lambda_0(\widetilde X)t}. \end{equation}
\subsection{Dealing with $a^{(1)}$} For fixed $(R,t)$, there are only finitely many $\gamma\in \Gamma$ such that $a_\gamma^{(1)}\neq 0$. Indeed, if $a_\gamma^{(1)}\neq 0$ then one can find $x,y\in \mathcal F$ such that $d(x,\gamma y)\le R$ and $y$ belongs to 
\[K:=\mathcal F\backslash \widetilde{\mathcal C}_R=\mathcal F\cap (\pi^{-1}(\overline B(o,R))),\]
where $\overline{B}(o,R)$ is the closed ball of radius $R$ centered at $o$ in $X$. Since $\mathcal F$ is a Dirichlet fundamental domain, we have $K\subset \overline{B}(\tilde o,R)$. Since $K$ is closed, it follows that it is also compact. By Lemma \ref{lem: finite number of terms}, this can only happen for $\gamma\in S\subset \Gamma$ where $S=S(K,R)$ is finite. Moreover, there is another compact $K'\subset \mathcal F$ such that for $(x,y)$ as above, one has $x\in K'$. This also shows that for all $\gamma$, the function $a_\gamma^{(1)}(x,y)$ is bounded and has bounded support, which implies that the associated integral operator $a_\gamma^{(1)}$ is compact. 

We can now apply the strong convergence result of Theorem \ref{prop: strong conv}: with probability tending to $1$ as the degree $n$ goes to $+\infty$, we have
\begin{equation} \label{eq: apply BC}\Big\|\sum_{\gamma\in \Gamma} a_\gamma^{(1)} \otimes \rho_\phi(\gamma)\Big\|_{L^2(\mathcal F)\otimes V_n^0}\le \Big\|\sum_{\gamma\in \Gamma} a_\gamma^{(1)} \otimes \rho_\infty(\gamma)\Big\|_{L^2(\mathcal F)\otimes \ell^2(\Gamma)}+\varepsilon. \end{equation}
We are left with estimating the operator norm on the right-hand side of \eqref{eq: apply BC}. Following \cite[\S 6.2]{HM}, there is an isomorphism of Hilbert spaces
\[\left\{\begin{array}{ll}L^2(\mathcal F)\otimes \ell^2(\Gamma)\simeq L^2(\widetilde X) \\
f\otimes \delta_\gamma \mapsto f\circ \gamma^{-1}\end{array}\right.,\]
and under this identification,
\[\sum_{\gamma\in \Gamma} a_\gamma^{(1)} \otimes \rho_\infty(\gamma)\simeq B ,\]
where $B:L^2(\widetilde X)\to L^2(\widetilde X)$ is a continuous operator with kernel
\[B(x,y)=H_{\widetilde X}(t,x,y)\mathbf 1_{d(x,y)\le R}(1-\chi_{\widetilde{\mathcal C}_R})(y).\]
To control the operator norm of $B$, observe that
\[0\le B(x,y)\le  H_{\widetilde X}(t,x,y)=\mathrm{e}^{-t\Delta_{\widetilde X}}(x,y).\]
This implies the inequality between the operator norms
\[\|B\|_{L^2(\widetilde X)}\le \|\mathrm{e}^{-t\Delta_{\widetilde X}}\|_{L^2(\widetilde X)}=\mathrm{e}^{-t\lambda_0(\widetilde X)}.\]
For convenience, we may take $\varepsilon=\mathrm{e}^{-t\lambda_0(\widetilde X)}$ in \eqref{eq: apply BC}. We obtain that with probability tending to $1$ as $n\to +\infty$, we have

\begin{equation} \label{eq: bound a^{(1)}} \Big\|\sum_{\gamma\in \Gamma} a_\gamma^{(1)} \otimes \rho_\phi(\gamma)\Big\|_{L^2(\mathcal F)\otimes V_n^0}\le 2\mathrm{e}^{-t\lambda_0(\widetilde X)}.\end{equation}

\subsection{Gathering everything} Combining \eqref{eq: bound op norm after splitting}, \eqref{eq: bound a^(2)}, \eqref{eq: bound a^{(1)}}, we find that for all fixed triplets $(R,t,r)$ with $R\ge t^2+R_0$ large enough (depending on $r$), with probability tending to $1$ as the degree $n\to +\infty$, one has
\[\big\|\exp(-t\Delta_{X_{\phi}})\big\|_{L^2_{\rm new}(X_\phi)}\le 4\mathrm{e}^{-t\lambda_0(\widetilde X)}+\mathrm{e}^{-t\lambda_0(\mathcal C_r)}.\]
Now, by \cite{DonnellyLi79}, we know that $\lambda_0(\mathcal C_r)\to \lambda_{\rm ess}(X)$ when $r\to +\infty$. Moreover, since $X$ is geometrically finite, by \cite[Theorem A]{BallmannEssentialBottom}, we have $\lambda_{\rm ess}(X)\ge \lambda_0(\widetilde X)$. Hence, for any $\varepsilon>0$, we can take $r$ large enough above to ensure 
\[4\mathrm{e}^{-t\lambda_0(\widetilde X)}+\mathrm{e}^{-t\lambda_0(\mathcal C_r)}\le 5\mathrm{e}^{-t(\lambda_0(\widetilde X)-\varepsilon/2)}.\]
In turn, $R$ and $t$ can be taken large enough (depending on $r$ and $\varepsilon$) to ensure $5< \mathrm{e}^{t\varepsilon/2}$. This shows that for any $\varepsilon>0$, there is some $t>0$ such that with probability tending to $1$ as $n\to +\infty$, we have
\[\big\|\exp(-t\Delta_{X_{\phi}})\big\|_{L^2_{\rm new}(X_\phi)}<\mathrm{e}^{-t(\lambda_0(\widetilde X)-\varepsilon)}.\]
This implies that, with probability tending to $1$ as $n\to +\infty$, the operator $\Delta_{X_\phi}|_{L^2_{\rm new}(X_\phi)}$ has no spectrum below $\lambda_0(\widetilde X)-\varepsilon$. This concludes the proof of Theorem \ref{thm: main}.

\printbibliography

\end{document}